\documentclass[a4paper,oneside,10pt]{amsart}

\usepackage{a4wide}
\usepackage[T1]{fontenc}
\usepackage[ansinew]{inputenc}
\usepackage{lmodern} 
\usepackage{graphicx}
\usepackage{amsmath}
\allowdisplaybreaks
\usepackage{amsthm}
\usepackage{amsfonts}
\usepackage{amssymb}
\usepackage{setspace}
\usepackage{mathrsfs}
\usepackage[all]{xy}
\usepackage{enumerate}
\usepackage{xcolor}
\usepackage{catchfile}
\usepackage{autobreak}

\theoremstyle{plain}

\newtheorem{theorem}{Theorem}[section]

\newtheorem{corollary}[theorem]{Corollary}
\newtheorem{proposition}[theorem]{Proposition}
 \newtheorem{lemma}[theorem]{Lemma}

\theoremstyle{definition}
\newtheorem{remark}[theorem]{Remark}

\newtheorem{assumption}[theorem]{Assumption}

 \newtheorem{example}[theorem]{Example}
 \newtheorem{definition}[theorem]{Definition}

\newtheorem*{proposition*}{Proposition}
\newtheorem*{definition*}{Definition}
\numberwithin{equation}{section}
\theoremstyle{plain}
\newtheorem*{theorem*}{Theorem}
\newtheorem*{hypo*}{Hypothesis}

\newenvironment{abc}{\begin{enumerate}[{\rm (a)}]}{\end{enumerate}}
\newenvironment{num}{\begin{enumerate}[{\rm 1.}]}{\end{enumerate}}
\newenvironment{iiv}{\begin{enumerate}[{\rm (i)}]}{\end{enumerate}}

\def\dom{\mathrm{D}}

\def\semis{\mathscr{P}}
\def\dd{\mathrm{d}}
\def\res{\mathrm{R}}

\def\ee{\mathrm{e}}
\def\RR{\mathbb{R}}

\def\NN{\mathbb{N}} 
\def\LLL{\mathscr{L}}

\def\Ell{\mathrm{L}}

\def\BC{\mathrm{C}_{\mathrm{b}}}
\def\BC{\mathrm{C}_{\mathrm{b}}}
\def\tlim{\mathop{\tau\lim}}

\def\R{\mathcal{R}}
\def\xX{\underline{X}}
\def\TT{\underline{T}}
\def\aA{\underline{A}}
\def\CC{\mathbb{C}}

\makeatletter
\newenvironment{proofof}[1]{\par
	\pushQED{\qed}%
	\normalfont \topsep6\p@\@plus6\p@\relax
	\trivlist
	\item[\hskip\labelsep
				\bfseries
		Proof of #1\@addpunct{.}]\ignorespaces		
}{%
	\popQED\endtrivlist\@endpefalse
}
\makeatother

\begin{document}
\title{Positive Desch--Schappacher perturbations of bi-continuous semigroups {on $\mathrm{AM}$-spaces}} 

\author{Christian Budde}
\address{North-West University, School of Mathematical and Statistical Sciences, Private Bag X6001-209, Potchefstroom 2520, South Africa}
\email{christian.budde@nwu.ac.za}

\begin{abstract}                                                                         
In this paper, we consider positive Desch--Schappacher perturbations of bi-continuous semigroups on $\mathrm{AM}$-spaces with an additional property concerning the additional locally convex topology. As an examples, we discuss perturbations of the left-translation semigroup on the space of bounded continuous function on the real line and on the space of bounded linear operators.
\end{abstract}


\date{}
\maketitle

\section*{Introduction}
Dynamical processes occurring, e.g., in population models, quantum mechanics, or the financial world, are frequently expressed by a particular class of partial differential equations, the so-called evolution equations. A general operator theoretical method for dealing with those equations is the one using abstract Cauchy problems on a Banach space $X$, i.e., one rewrites the partial differential equation as follows
\[
\begin{cases}\tag{ACP}\label{eqn:ACP}
\dot{u}(t)=Au(t),&\quad t\geq0,\\
u(0)=x\in X,
\end{cases}
\]
for some (unbounded) linear operator $(A,\dom(A))$ on $X$. This approach was worked out in detail by several authors, e.g., Engel and Nagel \cite{EN}, Pazy \cite{P1983} and Goldstein \cite{G2017}, just to mention a few. Classical solutions of \eqref{eqn:ACP} in the sense of \cite[Chapter II, Def.~6.1]{EN}, if they exist, can be represented by strongly continuous one-parameter semigroups of linear operators on $X$, or $C_0$-semigroups for short. By the Hille--Yosida theorem \cite[Chapter II, Thm.~3.8]{EN} one has such a classical solution if and only if $(A,\dom(A))$ is the generator of such an operator semigroup. However, this theorem consists of technical conditions that are difficult to verify for a concrete operator. In some cases, it is possible to write the given operator $(A,\dom(A))$ as a sum of simpler operators and this is where perturbation theory enters the area of evolution equations. The general question is: given a generator $(A,\dom(A))$ and another linear operator $(B,\dom(B))$, under which conditions does the operator $A+B$ generate a semigroup? Perturbations results for positive $C_0$-semigroups arouse a great deal of interest in the past, see for example \cite{PertPosAppl2006,Desch,ABE2014,ArendtRhandi1991,Voigt1977}.

\medskip
When talking about one-parameter semigroups of linear operators on Banach spaces, mostly $C_0$-semigroups come to mind. Nevertheless, there are operator semigroups which are not strongly continuous with respect to the norm on the Banach space but for some weaker additional locally convex topology on $X$. This is one of the reasons why people are interested in different continuity concepts of semigroups and more general solutions in order to overcome these limitations of strongly continuous semigroups. There were different attempts: integrable semigroups of Kunze \cite{K2009}, "C-class" semigroups of Kraaij \cite{Kr2016}, $\pi$-semigroups of Priola \cite{P1999} and weakly continuous semigroups of Cerrai \cite{C1994}, to mention a few. These kind of operator semigroups concentrate predominantly on the space $\mathrm{C}_{\mathrm{b}}(E)$ of bounded continuous functions on a metric space $(E,d)$ and the space $\mathrm{C}_{\mathrm{ub}}(\mathcal{H})$ of bounded uniformly continuous functions on a Hilbert space $\mathcal{H}$ which occur naturally in the context of transition semigroups. The general idea is to consider operator semigroups with respect to a locally convex topology that is weaker than the norm topology. Nonetheless, these semigroups do not cover every important example of non-strongly continuous semigroups on Banach spaces in a uniform manner, e.g., adjoint semigroups \cite{vN1992}, implemented semigroups \cite{Alber2001}, semigroups induced by flows, the Ornstein--Uhlenbeck semigroup or Markov processes, just to mention a few (see for example \cite{DN1996,DPL1995,vN1992,LB2007} or \cite[Chapter 3]{KuPhD}). One of the auspicious approaches to this gives rise to so-called bi-continuous semigroups, which were introduced by K\"uhnemund \cite{KuPhD,Ku}. The idea is to equip the underlying Banach space with an additional locally convex topology $\tau$ according to some requirements on the interaction with the norm topology. The perturbation theory of such semigroups was developed by Farkas \cite{FaPhD,FaSF,FaStud} and more recently by Farkas and the author \cite{BF,BF2}, see also \cite{BuPosMiy,BuddePhD}. That one takes also positivity of the semigroup into account is for example motivated by real-life applications where most of the quantities one works with have to be positive in order to make sense in the specific context, e.g., total temperature or density. A first perturbation result for positive bi-continuous semigroups has recently been proved by the author \cite{BuPosMiy} by means of positive Miyadera--Voigt perturbations, whereas the general notion of positivity in combination with bi-continuous semigroups makes its first appearance in \cite{ESF2005}.

\medskip
In this paper, we discuss positive Desch--Schappacher perturbations of bi-continuous semigroups. {In order to talk about positivity, the underlying Banach space needs more structure, i.e., one needs an ordering which allows comparing elements of the Banach space. For this purpose, we will consider a special class of Banach lattices, the so-called $\mathrm{AM}$-spaces. They have been extensively studied by Kakutani \cite{K1941} who showed that every $\mathrm{AM}$-space has a concrete representation as a subspace of the space $\mathrm{C}(\Omega)$ of all continuous real valued functions defined on a topological compact Hausdorff space $\Omega$ by means of lattice isomorphisms. Furthermore, we will see that the $\mathrm{AM}$-spaces offer a particularly rewarding setting for our theory. Desch--Schappacher type perturbations} for bi-continuous semigroups have been discussed by the author and Farkas in \cite{BF2}. However, positivity in the context of bi-continuous Desch--Schappacher perturbation was not considered yet. Next to the pure theoretical interest of such a theorem, the usage of such a result is motivated by applications in the theory of PDEs. To be more specific, if one considers Markov processes coming from stochastic differential equations, cf. \cite[Sect.~2.4 \& 2.5]{LB2007}, one is interested in positive solutions of the corresponding parabolic problem associated with differential operators, on the space of bounded continuous functions $\BC(\RR^n)$. Moreover, this space is of special interest since operator semigroups associated with Markov processes, in general, do not leave the space $\mathrm{C}_0(\RR^n)$ invariant, cf. \cite[Sect.~5.2 \& 5.3]{LB2007}. Furthermore, Desch--Schappacher type perturbations are of particular interest since there is a wide range of application, e.g., domain and boundary perturbations by Greiner \cite{Greiner1987}, Nickel \cite{Nickel2004} or Hadd, Manzo and Rhandi \cite{Rhandi2014}, boundary control by Engel, Kramar Fijav\v{z}, Kl\"{o}ss, Nagel and Sikolya \cite{EKKNS2010} and Engel and Kramar Fijav\v{z} \cite{EK2017} or control theory by Jacob, Nabiullin, Partington and Schwenninger \cite{JNPS2018,JNPS2016} and Jacob, Schwenninger and Zwart \cite{JSZ2018}, just to mention a few here. 

\medskip
As main reference serves the work of A.~Batk\'{a}i, B.~Jacob, J.~Wintermayr and J.~Voigt \cite{BJVW2018} on positive Desch--Schappacher perturbations of strongly continuous operator semigroups. Nevertheless, the techniques used in this paper differ from \cite{BJVW2018}. Firstly, one has to take the additional locally convex topology into account. Secondly, the construction of the extrapolation spaces for bi-continuous semigroups, which are essential for Desch--Schappacher type perturbations, is significantly more complicated than in the case of strongly continuous semigroups \cite{BF}. Moreover, one can cover a spectrum of applications beyond the $C_0$-semigroup setting. In particular, we will consider the translation semigroup on the space of bounded continuous functions $\BC(\RR)$ as well as the implemented semigroup on the space of bounded linear operators $\LLL(E,F)$ between two Banach spaces $E$ and $F$. The latter semigroup appears naturally in the context of Wigner symmetries in quantum mechanics, see for example \cite[Chapter 5]{Landsman2017} or \cite[Chapter 3, Section 2]{bratteli2013operator}.

\medskip
The first two sections have a preliminary character. Firstly, we introduce the concept of bi-continuous semigroups and recall some facts about ordering in Banach spaces. In the second section, we combine extrapolation spaces and positivity. Section 3 consists of the main perturbation result as well as its proof. In the final section, we consider examples of perturbations on the spaces $\BC(\RR)$ and $\LLL(E,F)$.

\section{Preliminaries}
\subsection{Bi-continuous semigroups}
The theory of bi-continuous semigroups was first introduced by K\"uhnemund \cite{KuPhD,Ku}. K\"uhnemund suggested the following assumptions, which are related to Saks spaces, cf. \cite{C1987}.

\begin{assumption}\label{asp:bicontspace} 
Consider a triple $(X,\|\cdot\|,\tau)$ where $\left(X,\|\cdot\|\right)$ is a Banach space, and
\begin{num}
\item $\tau$ is a locally convex Hausdorff topology coarser than the norm-topology on $X$, i.e., the identity map $(X,\|\cdot\|)\to(X,\tau)$ is continuous;
\item $\tau$ is sequentially complete on the $\left\|\cdot\right\|$-closed unit ball, i.e., every $\left\|\cdot\right\|$-bounded $\tau$-Cauchy sequence is $\tau$-convergent;
\item The dual space of $(X,\tau)$ is norming for $X$, i.e.,
\begin{equation}\label{eq:norm}
\|x\|=\sup_{\substack{\varphi\in(X,\tau)'\\\|\varphi\|\leq1}}{|\varphi(x)|},\quad x\in X.\end{equation}
\end{num}
\end{assumption}

\begin{remark}\label{rem:examAssump}
Every locally convex topology $\tau$ is induced by a family of continuous seminorms $\semis$ and vice versa, cf. \cite[Chapter II, Sect.~4]{Schaefer1971} or \cite[Thm.~1.36 \& 1.37]{Rudin}. In particular, if we want to make explicit calculations within the framework of locally convex topologies, we use the corresponding seminorms.
\end{remark}

\begin{example}\label{exa:BiContSpaces}
\begin{abc}
	\item The space $\BC(\RR)$ of bounded continuous functions on $\RR$ equipped with the compact-open topology is a typical example of satisfying the previous assumptions. 
	\item The space $\LLL(E,F)$ of bounded linear operators between Banach spaces $E$ and $F$ together with the strong operator topology satisfies Assumption \ref{asp:bicontspace}.
	\item Likewise, the essentially bounded functions $\Ell^{\infty}(\RR)$ with the weak$^*$-topology, as it is the dual space of $\mathrm{L}^1(\RR)$, belongs to the list of examples in this context. 
	\item More general, the dual space $X'$ of a Banach space $X$ equipped with the weak$^*$-topology satisfies Assumption \ref{asp:bicontspace}.	For more details and examples we refer to \cite[Chapter 1]{BuddePhD}.
\end{abc}
\end{example}

Now we introduce bi-continuous semigroups as it was done by K\"uhnemund.

\begin{definition}\label{def:bicontsemi}
Let $X$ be a Banach space with norm $\|\cdot\|$ together with a locally convex topology $\tau$ such that the conditions in Assumption \ref{asp:bicontspace} are satisfied. We call $(T(t))_{t\geq0}$ a \emph{$\tau$-bi-continuous semigroup} if
\begin{num}
\item $ T(t+s)=T(t)T(s)$ and $T(0)=I$ for all $s,t\geq 0$,
\item $(T(t))_{t\geq0}$ is strongly $\tau$-continuous, i.e. the map $\varphi_x:[0,\infty)\to(X,\tau)$ defined by $\varphi_x(t)=T(t)x$ is $\tau$-continuous for every $x\in X$,
\item $(T(t))_{t\geq0}$ has type $(M,\omega)$ for some $M\geq 1$ and $\omega\in \RR$, i.e., $\left\|T(t)\right\|\leq M\ee^{\omega t}$ for all $t\geq0$,
\item $(T(t))_{t\geq0}$ is locally-bi-equicontinuous, i.e., if $(x_n)_{n\in\NN}$ is a norm-bounded sequence in $X$ which is $\tau$-convergent to $0$, then also $(T(s)x_n)_{n\in\NN}$ is $\tau$-convergent to $0$ uniformly for $s\in[0,t_0]$ for each fixed $t_0\geq0$.
\end{num}
\end{definition}

Similarly to the case of $C_0$-semigroups, one defines the generator of a bi-continuous semigroup as follows.

\begin{definition}\label{def:BiGen}
Let $(T(t))_{t\geq0}$ be a $\tau$-bi-continuous semigroup on $X$. The \emph{(infinitesimal) generator} of $(T(t))_{t\geq0}$ is the linear operator $(A,\dom(A))$ defined by
\[Ax:=\tlim_{t\to0}{\frac{T(t)x-x}{t}}\] with domain 
\[\dom(A):=\Bigl\{x\in X:\ \tlim_{t\to0}{\frac{T(t)x-x}{t}}\ \text{exists and} \ \sup_{t\in(0,1]}{\frac{\|T(t)x-x\|}{t}}<\infty\Bigr\}.\]
\end{definition}

As a matter of fact, the generator of a bi-continuous semigroup behaves almost like a $C_0$-semigroup generator. Since we do not need all of the properties here, we refer for more details to \cite[Prop.~1.16 \& Prop.~1.18]{KuPhD}. Furthermore, there exists a Hille--Yosida generation type theorem for bi-continuous semigroups \cite[Thm.~16]{Ku}.

\subsection{Ordered topological vector spaces}
We now recall some basic definitions of ordered topological vector spaces and operators acting on such spaces, see \cite{Schaefer1974} or \cite{GvR2016} for further background material. First of all, we recall the notion of vector and Banach lattices. Notice that for our purpose only real vector spaces are of interest.

\begin{definition}
A \emph{vector lattice} or \emph{Riesz space} $(V,\leq)$ is a vector space $V$ equipped with a partial order $\leq$ such that for each $x,y,z\in V$:
\begin{abc}
	\item $x\leq y \Rightarrow x+z\leq y+z$.
	\item $x\leq y \Rightarrow \alpha x\leq\alpha y$ for all scalars $\alpha\geq0$.
	\item For any pair $x,y\in V$ there exists a supremum, denoted by $x\vee y$, and an infimum, denoted by $x\wedge y$, in $V$ with respect to the partial order $\leq$.
\end{abc}
An element $x\in V$ is called \emph{positive} if $x\geq0$. The set of all positive elements of $V$ is denoted by $V_+$. Furthermore, the \emph{absolute value} of an element $x\in X$ is defined by $\left|x\right|:=x\vee(-x)$. {Furthermore, $x_+:=x\vee0$ and $x_-:=x\wedge 0$ are called the \emph{positive part} and \emph{negative part}, respectively.}
\end{definition}

\begin{definition}
A \emph{Banach lattice} is a triple $(X,\left\|\cdot\right\|,\leq)$ where $(X,\left\|\cdot\right\|)$ is a Banach space equipped with a partial order $\leq$ such that $(X,\leq)$ is a Riesz space satisfying $\left|x\right|\leq\left|y\right| \Rightarrow \left\|x\right\|\leq\left\|y\right\|$ for all $x,y\in X$.
\end{definition}

The following definition also incorporates the locally convex topology $\tau$ which belongs into the framework of bi-continuous semigroups. The concept is related to the work of Kawai \cite[Thm.~1]{K1957}.

\begin{definition}\label{def:comp}
Let $(X,\left\|\cdot\right\|,\leq)$ and $\tau$ a locally convex topology satisfying Assumption \ref{asp:bicontspace} generated by a family of seminorms $\semis$. We say that $\tau$ is \emph{compatible (with the Banach lattice structure)} if for all $p\in\mathscr{P}$ and all $x,y\in X$ one has that $p\left(x\right)\leq p\left(y\right)$ whenever $\left|x\right|\leq\left|y\right|$.
\end{definition}

Next we introduce a special class of $\mathrm{AM}$-spaces which are related to the $\mathrm{AM}$-spaces used by Batkai et al. \cite[Rem.~2.5]{BJVW2018}. In fact, one needs an additional requirement on the interaction of the partial order and the seminorms.

\begin{definition}
Let $(X,\left\|\cdot\right\|,\leq)$ be a Banach lattice and $\tau$ a locally convex topology satisfying Assumption \ref{asp:bicontspace}. Let $\semis$ be the family of seminorms generating $\tau$. Then $X$ is called a \emph{bi-$\mathrm{AM}$-space} if $\sup\left\{\left\|x\right\|,\left\|y\right\|\right\}=\left\|x\vee y\right\|$ and $\sup\left\{p(x),p(y)\right\}=p\left(x\vee y\right)$ for all $x,y\in X_+$ and $p\in\semis$.
\end{definition}

Now we take also linear operators on such lattices into consideration. Let $(X,\left\|\cdot\right\|,\leq)$ be a Banach lattice and $T:X\to X$ a bounded linear operator. Then $T$ is called \emph{positive}, denoted by $T\geq0$, if $Tx\geq0$ for each $x\in X_+$. A semigroup of bounded linear operators $(T(t))_{t\geq0}$ on such a Banach lattice is called \emph{positive} if $T(t)\geq0$ for each $t\geq0$. 

\medskip
In order to make sense of the following notion, recall that for a linear operator $(A,\dom(A))$ on a real Banach space $X$ the \emph{resolvent set} $\rho(A)$ consists of $\lambda\in\mathbb{C}$ such that $\lambda-A$ is invertible on the complexification $X_\mathbb{C}$ of $X$ defined by $X_\mathbb{C}=X\otimes_{\RR}\mathbb{C}$, i.e., there exists a bounded operator $B$ with range equal to $\dom(A)$ such that $(\lambda-A)Bx=x$ for each $x\in X_\CC$ and $B(\lambda-A)x=x$ for each $x\in\dom(A)$. The complement of $\rho(A)$, is called the \emph{spectrum} and is denoted by $\sigma(A):=\mathbb{C}\setminus\rho(A)$. The so-called \emph{spectral radius} of an operator $(A,\dom(A))$ is defined by $\mathrm{r}(A):=\sup\left\{\left|\lambda\right|:\ \lambda\in\sigma(A)\right\}$. Also, recall that the \emph{spectral bound} $s(A)$ of an operator $(A,\dom(A))$ is defined by $s(A):=\sup\left\{\mathrm{Re}(\lambda):\ \lambda\in\sigma(A)\right\}$. The inverse of $\lambda-A$, called the \emph{resolvent}, is often denoted by $R(\lambda,A):=(\lambda-A)^{-1}$. If $(A,\dom(A))$ is the generator of a bi-continuous semigroup $(T(t))_{t\geq0}$, then it was shown that the resolvent can be expressed as a Laplace transform, cf. \cite[Sect.~1.2]{Ku} or \cite[Thm.~1.2.7]{FaPhD}. To be exact, for $\lambda>\omega_0$ one has $\lambda\in\rho(A)$ and
\[
R(\lambda,A)x=\int_0^{\infty}{\ee^{-\lambda t}T(t)x\ \dd{t}},
\]
for all $x\in X$, where the integral is a $\tau$-improper integral and $\omega_0$ is the growth bound defined to be the infimum over $\omega\in\RR$ appearing in the exponential norm estimate in Definition \ref{def:bicontsemi}(3), cf. \cite[Sect.~1.1]{KuPhD} or \cite[Sect.~1.2]{FaPhD}. The following definition related to the connection between positivity and unbounded operators was suggested by W.~Arendt \cite{Arendt1987}.

\begin{definition}
A linear operator $(A,\dom(A))$ on a Banach lattice $(X,\left\|\cdot\right\|,\leq)$ is called \emph{resolvent positive} if there exists $\omega\in\RR$ such that $(\omega,\infty)\subseteq\rho(A)$ and such that $R(\lambda,A)\geq0$ for each $\lambda>\omega$.
\end{definition}

By \cite[Cor.~2.3]{Arendt1987} a $C_0$-semigroups on a Banach lattice is positive if and only if the corresponding generator $(A,\dom(A))$ is resolvent positive. The same result holds true for the case of bi-continuous semigroups, cf. \cite[Thm.~1.4.1]{FaPhD}.

\section{Extrapolation spaces and positivity}

We first recall some results concerning extrapolation spaces for bi-continuous semigroups from \cite{BF} and \cite{BF2}. Throughout this section we assume without loss of generality that $0\in\rho(A)$. One of the most important ingredients for the construction of extrapolation spaces is the following proposition, cf. \cite[Prop.~1.18]{KuPhD} and \cite[Chapter II, Cor.~3.21]{EN}

\begin{proposition}\label{prop:StrCont}
Let $(T(t))_{t\geq0}$ be a bi-continuous semigroup on $X$ with generator $(A,\dom(A))$. The subspace $\xX_0:=\overline{\dom(A)}^{\left\|\cdot\right\|}\subseteq X$ is $(T(t))_{t\geq0}$-invariant and $(\TT(t))_{t\geq0}:=(T(t)_{|\xX_0})_{t\geq0}$ is the $C_0$-semigroup on $\xX_0$ generated by the part of $A$ in $\xX_0$ (this generator is denoted by $\aA_0$). 
\end{proposition}

This construction, elementary properties and examples can be found in \cite[Chapter II, Sect.~5a]{EN}, \cite{NagelIdent} or \cite{N1997}. Without loss of generality, we may assume that $0\in\rho(A)$. Recall from \cite[Chapter II, Def.~5.4]{EN} that one obtains $\xX_{-1}$ as the completion of $\xX_0$ with respect to the $\left\|\cdot\right\|_{-1}$-norm defined by
\[
\left\|x\right\|_{-1}:=\left\|\aA_0^{-1}x\right\|,\quad \text{for }x\in\xX_0.
\]
Notice that $\xX_0$ is dense in $\xX_{-1}$ and that $(\TT(t))_{t\geq0}$ extends by continuity to a $C_0$-semigroup $(\TT_{-1}(t))_{t\geq0}$ on $\xX_{-1}$ with generator $(\aA_{-1},\dom(\aA_{-1}))$, where $\dom(\aA_{-1})=\xX_0$. By repeating this construction one obtains the following chain of spaces
\[
\xX_0\stackrel{\aA_{-1}}{\hookrightarrow}\xX_{-1}\stackrel{\aA_{-2}}{\hookrightarrow}\xX_{-2}\rightarrow\cdots
\]
where all inclusions are continuous and dense. Moreover, we have
\[
\xX_0\hookrightarrow X\hookrightarrow\xX_{-1},
\]
see \cite{BF}, so that we can identify $X$, by the continuity of the inclusions, as a subspace of $\xX_{-1}$. This leads to the following definition of extrapolation spaces for bi-continuous semigroups, cf. \cite[Prop.~5.8]{BF}.

\begin{definition}
The \emph{(first) extrapolation space} $X_{-1}$ for a bi-continuous semigroup $(T(t))_{t\geq0}$ with generator $(A,\dom(A))$ is defined by
\begin{align*}
X_{-1}:=\aA_{-2}(X).
\end{align*}
The norm on $X_{-1}$ is defined by $\|x\|_{-1}:=\|\aA^{-1}_{-2}x\|$, the locally convex topology $\tau_{-1}$ on $X_{-1}$ is induced by the family of seminorms $\semis_{-1}:=\left\{p_{-1}:\ p\in\semis\right\}$ where
\[
p_{-1}(x):=p(\aA_{-2}^{-1}x),\quad p\in\semis,\:x\in X_{-1}.
\]
\end{definition}

\begin{remark}\label{rem:EquivExtraSp}
As mentioned previously, we assume without loss of generality that $0\in\rho(A)$, in particular, $A^{-1}$ exists and is a bounded linear operator. However, one can also use any other operator $R(\lambda,A)$, $\lambda\in\rho(A)$, since the constructed spaces are all equivalent.
\end{remark}

It was shown in \cite[Prop.~5.8]{BF} that $(T(t))_{t\geq0}$ extends to a $\tau_{-1}$-bi-continuous semigroup $(T_{-1}(t))_{t\geq0}$ on $X_{-1}$ and has a generator $A_{-1}$ with domain $\dom(A_{-1})=X$. The operator $A_{-1}:X\to X_{-1}$ is an isomorphism intertwining the semigroups $(T(t))_{t\geq0}$ and $(T_{-1}(t))_{t\geq0}$.

\medskip
Before we equip the extrapolation space with a lattice structure, we need the concept of the so-called \emph{mixed topology}. The mixed topology has been introduced by Wiweger \cite{W1961} for general topological vector spaces. However, we will restrict us to the simpler situation where we consider the $\left\|\cdot\right\|$-topology and the additional locally convex topology $\tau$ satisfying Assumption \ref{asp:bicontspace}. We construct the mixed topology by following \cite[Thm.~3.1.1]{W1961} and \cite[Sect.~A.1]{FaPhD}. For $p_n\in\semis$ and $(a_n)_{n\in\NN}\in\mathrm{c}_0$, $a_n\geq0$, $n\in\NN$, one defines the seminorms on $X$ by
\[
\widetilde{p}_{(a_n,p_n)}(x):=\sup_{n\in\NN}{a_np_n(x)},\quad x\in X.
\]

\begin{definition}
The \emph{mixed topology} associated to a triple $(X,\left\|\cdot\right\|,\tau)$ satisfying Assumption \ref{asp:bicontspace} is the locally convex topology induced by the family of seminorms 
\begin{align}\label{eqn:SemiMix}
\widetilde{\semis}:=\left\{\widetilde{p}_{(a_n,p_n)}:\ p_n\in\semis,\ (a_n)_{n\in\NN}\in\mathrm{c}_0,\ a_n\geq0\right\}
\end{align}
and will be denoted by $\gamma:=\gamma(\tau,\left\|\cdot\right\|)$.
\end{definition}

It is obvious, that $\tau\subseteq\gamma\subseteq\left\|\cdot\right\|$. Moreover, by \cite[Lemma~A.1.2]{FaPhD} a sequence $(x_n)_{n\in\NN}$ in $X$ converges to $x$ with respect to $\gamma$ if and only if the sequence in $\left\|\cdot\right\|$-bounded and $\tau$-convergent to $x$. This actually implies the following result, cf. \cite[Prop.~A.1.3]{FaPhD}.

\begin{proposition}
The class of bi-continuous semigroups coincides with the class of $\gamma$-strongly continuous and locally sequentially $\gamma$-equicontinuous semigroups.
\end{proposition}

Now, we are able to introduce a lattice structure on extrapolation spaces. As the lattice structure of $X$ is not naturally hereditary to the extrapolation spaces one needs a new concept. For $C_0$-semigroups, this was introduced by Batkai et al. \cite[Def.~2.1]{BJVW2018} and \cite[Def.~2.1]{W2016}. Following this model, one could think about defining $x\in X_{-1}$ to be positive if and only if there exists a sequence $(x_n)_{n\in\NN}$ in $X_+$ such that $(x_n)_{n\in\NN}$ is $\left\|\cdot\right\|_{-1}$-bounded and $x_n\stackrel{\tau_{-1}}{\rightarrow}x$. However, by \cite[Lemma~A.1.2]{FaPhD} this would coincide with the $\gamma$-sequential closure of $X_+$. In general, the sequential closure is not sequentially closed, see for example \cite[Rem.~1(ii)]{W2014}. As we introduced the mixed topology, we denote by $\gamma_{-1}:=\gamma(\tau_{-1},\left\|\cdot\right\|_{-1})$ the mixed topology on the extrapolation space $X_{-1}$. 
In what follows, we also need that the considered semigroup is not only sequentially $\gamma$-locally equicontinuous but $\gamma$-locally equicontinuous. Therefore, we will make the following standing assumptions.

\begin{assumption}\label{ass:CompMixed}
The space $(X,\gamma)$ is complete and $(T(t))_{t\geq0}$ is $\gamma$-locally equicontinuous, i.e., for some $t_0>0$ (or equivalently for all $t_0>0$) holds that
\[
\forall p\in\widetilde{\semis}\ \exists q\in\widetilde{\semis}\ \exists M\geq0\ \forall x\in X\ \forall t\in\left[0,t_0\right]: p(T(t)x)\leq Mq(x). 
\]
\end{assumption}

\begin{remark}
\begin{abc}
	\item The use of Assumption \ref{ass:CompMixed} is twofold. First of all, due to the completeness of $(X,\gamma)$, one is able to extrapolate the mixed topology $\gamma$ and in the view of \cite[Thm.~1 \& Rem.~1(i)]{W2014} one obtains that $\gamma_{-1}$ coincides with the extrapolated locally convex topology of $\gamma$. Secondly, one does not struggle with sequential closures since, as mentioned already above, in general, the sequential closure is not sequentially closed.
	\item Of course, Assumption \ref{ass:CompMixed} is somehow restrictive, nevertheless, it is still reasonable, see for example \cite{GK2001} and \cite{Kr2016}. In particular, Farkas showed that bi-continuous semigroups on the space $\BC(\Omega)$, for $\Omega$ a Polish space, are automatically local and hence $\gamma$-locally equicontinuous, cf. \cite[Prop.~3.3 \& Thm.~3.4]{FCz2011}.
\end{abc}
\end{remark}

\begin{definition}\label{def:PosConeExtra}
Let $(X,\left\|\cdot\right\|,\leq)$ be a Banach lattice with positive cone $X_+$ and compatible locally convex topology $\tau$. Let $X_{-1}$ be the extrapolation space corresponding to the positive bi-continuous semigroup $(T(t))_{t\geq0}$. {We say that $x\in X_{-1}$ is \emph{positive}, if $x$ belongs to the $\gamma_{-1}$-closure of $X_+$ in $X_{-1}$. By $X_{-1,+}$ we denote all positive elements of $X_{-1}$.}
\end{definition}

\begin{proposition}\label{prop:PosCone}
Let $(X,\left\|\cdot\right\|,\leq)$ be a Banach lattice and $\tau$ a compatible locally convex topology and let $(T(t))_{t\geq0}$ be a positive bi-continuous semigroup on $X$ with generator $(A,\dom(A))$. The set $X_{-1,+}$ is a $\gamma_{-1}$-closed convex cone in $X_{-1}$ satisfying $X_+=X_{-1,+}\cap X$.
\end{proposition}

\begin{proof}
By taking the $\gamma_{-1}$-closure in the inclusions $X_++X_+\subseteq X_+$ and $\alpha X_+\subseteq X_+$ for $\alpha\geq0$ we obtain the corresponding inclusions for $X_{-1,+}$. The $\gamma_{-1}$-closedness of $X_{-1,+}$ follows directly from the construction. To conclude that $X_{-1,+}$ is a cone, we have to show that $X_{-1,+}\cap\left(-X_{-1,+}\right)=\left\{0\right\}$. To do so, assume that $x\in X_{-1,+}$ and $-x\in X_{-1,+}$. By definition, there exist nets $(x_\alpha)_{\alpha\in I}$ and $(y_\beta)_{\beta\in J}$ in $X_+$ such that $x_\alpha\stackrel{\gamma_{-1}}{\rightarrow}x$ and $y_\beta\stackrel{\gamma_{-1}}{\rightarrow}-x$. Obviously, we have $0\leq x_\alpha\leq x_\alpha+y_\beta$ for all $\alpha\in I$ and $\beta\in J$ and therefore {one gets $0\leq\res(\lambda,A)x_\alpha\leq\res(\lambda,A)(x_\alpha+y_\beta)$ for all for all $\alpha\in I$ and $\beta\in J$, whenever $\lambda>\omega$ for $\omega\in\RR$ such that $(\omega,\infty)\subseteq\rho(A)$, due to the fact that $(A,\dom(A))$ is the generator of a positive bi-continuous semigroup, cf. \cite[Thm.~1.4.1]{FaPhD}}. Whence by the compatibility of $\tau$ one gets
\[
\widetilde{p}_{-1,(a_n,p_n)}(x_\alpha)={p}_{(a_n,p_n)}(\res(\lambda,A)x_\alpha)\leq p_{(a_n,p_n)}(\res(\lambda,A)(x_\alpha+y_\beta))=\widetilde{p}_{-1,(a_n,p_n)}(x_\alpha+y_\beta)\to 0,
\]
for all $(a_n)_{n\in\NN}\in\mathrm{c}_0$ with $a_n\geq0$ and $p_n\in\semis$, see also \eqref{eqn:SemiMix}. From this we conclude that $x_\alpha\stackrel{\gamma_{-1}}{\rightarrow}0$ and hence $x=0$, showing that $X_{-1,+}\cap\left(-X_{-1,+}\right)=\left\{0\right\}$.
Finally, we show that the equality $X_+=X_{-1,+}\cap X$ holds. The inclusion $X_+\subseteq X_{-1,+}\cap X$ is obvious by definition, so we only have to show that the reverse inclusion holds. To do so, let $x\in X_{-1,+}\cap X$, i.e., there exists $(x_\alpha)_{\alpha\in I}$ in $X_+$ such that $x_\alpha\stackrel{\gamma_{-1}}{\rightarrow}x$. For $\lambda\in\rho(A)$ we obtain that $\res(\lambda,A)x_\alpha\stackrel{\gamma}{\to}\res(\lambda,A)x$. Since $\res(\lambda,A)x_\alpha\geq0$ for each $\alpha\in I$ and we conclude by the $\gamma_{-1}$-closedness that $\res(\lambda,A)x\geq0$. Since by \cite[Lemma~2.1(i)]{O1973} or \cite[Lemma~6.3]{Kr2016} also holds that $\lambda\res(\lambda,A)x\to x$ with respect to $\gamma$ for $\lambda\to\infty$ we conclude that $x\in X_+$.
\end{proof}

\section{Positive Desch--Schappacher Perturbations}

In this section we present and prove the main result of this paper, which is the following theorem. {The following definition is important, cf. \cite[Def.~1.2.20]{FaPhD}.

\begin{definition}
Let $T$ be a norm-bounded operator on a Banach space $X$. Suppose that $\semis$ is a family of seminorms generating the local convex topology $\tau$. We call $T$ \emph{local} if for all $p\in\semis$ and $\varepsilon>0$ there exist $K>0$ and $q\in\semis$ such that for all $x\in X$ one has
\[
p(Tx)\leq Kq(x)+\varepsilon\left\|x\right\|.
\]
\end{definition}
}
 
Before we formulate and prove our main result, we derive some preliminary results, starting with the following in the style of \cite[Lemma~4.3]{BJVW2018}.

\begin{lemma}\label{lemma:StepFunctions}
Let $\left(X,\left\|\cdot\right\|\right)$ be a Banach space with ordering $\leq$ and an additional locally convex topology $\tau$. Let $X_{-1}$ be the extrapolation space for the positive bi-continuous semigroup $(T(t))_{t\geq0}$ on $X$. Let $B\in\LLL(X,X_{-1})$ such that $B:\left(X,\tau\right)\to\left(X_{-1},\tau_{-1}\right)$ is continuous and positive. Then for all $t_0>0$ we have:
\begin{iiv}
	\item $(T_{-1}(t))_{t\geq0}$ is positive on $X_{-1}$.
	\item For each step function $u\in\Ell^{\infty}\left(\left[0,t_0\right],X\right)$ we have
	\[
	\int_0^{t_0}{T_{-1}(s)Bu(s)\ \dd{s}}\in X.
	\]
	\item For all $x\in X_+$ one has
	\[
	\int_0^{t_0}{T_{-1}(s)Bx\ \dd{s}}\in X_+.
	\]
	\item If $(T(t))_{t\geq0}$ is exponential stable{, i.e., there exists $M\geq1$ and $\omega>0$ such that $\left\|T(t)\right\|\leq M\ee^{-\omega t}$ for all $t\geq0$,} then 
	\[
	\int_0^{t_0}{T_{-1}(s)Bx\ \dd{s}}\leq\int_0^{\infty}{T_{-1}(s)Bx\ \dd{s}}
	\]
	in $X$ for all $x\in X_+$.
\end{iiv}
\end{lemma}

\begin{proof}
\begin{iiv}
	\item This follows immediately, since $(T_{-1}(t))_{t\geq0}$ is the continuous extension of $(T(t))_{t\geq0}$, for all $t\geq0$.	
	\item Let $u\in\Ell^{\infty}\left(\left[0,t_0\right],X\right)$ be a step function, i.e., there exist pairwise disjoint intervals $I_1,\ldots,I_N$ and $x_1,\ldots,x_N\in X$ such that 
	\[
	u(t)=\sum_{n=1}^N{x_n\textbf{1}_{I_n}(t)},\quad t\in\left[0,t_0\right],
	\]
	where $\textbf{1}_{I_k}$, $1\leq k\leq N$, denotes the characteristic function of $I_k$. Since the integral is additive we just need to show that for an interval $I_n:=\left[t_n,t_{n+1}\right]$ holds that
	\[
	\int_{I_n}{T_{-1}(s)Bx_n\ \dd{s}}=\int_{t_n}^{t_{n+1}}{T_{-1}(s)Bx_n\ \dd{s}}\in X.
	\]
	By using substitution we obtain that
	\[
	\int_{t_n}^{t_{n+1}}{T_{-1}(s)Bx_n\ \dd{s}}=\int_0^{t_{n+1}-t_n}{T_{-1}(s+t_n)Bx_n\ \dd{s}}=T_{-1}(t_n)\int_0^{t_{n+1}-t_n}{T_{-1}(s)Bx_n\ \dd{s}}.
	\]
	Since $(T_{-1}(t))_{t\geq0}$ is bi-continuous on $X_{-1}$ with generator $A_{-1}$ we obtain by \cite[Thm.~1.2.7(c)]{FaPhD} that
	\[
	\int_0^{t_{n+1}-t_n}{T_{-1}(s)Bu(s)\ \dd{s}}\in\dom(A_{-1})=X
	\]
	Part $\mathrm{(iii)}$ and $\mathrm{(iv)}$ directly follow from Proposition \ref{prop:PosCone}.
	\end{iiv}
\end{proof}

The following theorem is the main result and uses arguments similar to \cite[Prop.~4.2]{BJVW2018}. 

\begin{theorem}\label{prop:PosPertAux}
Let $\left(X,\left\|\cdot\right\|,\tau\right)$ be a bi-$\mathrm{AM}$-space, $(T(t))_{t\geq0}$ a positive bi-continuous semigroup on $X$ with generator $(A,\dom(A))$. Let $B\in\LLL(X,X_{-1})$ such that $B:(X,\tau)\to(X_{-1},\tau_{-1})$ is continuous. Suppose that $B$ is positive and that there exists a $\lambda>\mathrm{s}(A)$ such that $\res(\lambda,A_{-1})B$ is local and that $K:=\left\|\res(\lambda,A_{-1})B\right\|<1$. Then $(A_{-1}+B)_{|X}$ is the generator of a positive bi-continuous semigroup $(S(t))_{t\geq0}$.
\end{theorem}

To prove Theorem \ref{prop:PosPertAux}, we recall the Desch--Schappacher perturbation type theorem for bi-continuous semigroups \cite[Thm.~4.1]{BF2}.

\begin{theorem}\label{thm:admDS}
{Let $(T(t))_{t\geq0}$ be a $\tau$-bi-continuous semigroup with generator $(A,\dom(A))$ on a Banach space $X$.} Let $\mathscr{P}$ be the set of generating continuous seminorms corresponding to $\tau$. Let {$B\in \LLL(X,X_{-1})$} such that $B:(X,\tau)\rightarrow(X_{-1},\tau_{-1})$ is continuous, and let $t_0>0$ be such that
\begin{abc}
	\item $\displaystyle{\int\limits_0^{t_0}{T_{-1}(t_0-r)Bf(r)\ \dd r}\in X}$ for each $f\in\BC\left(\left[0,t_0\right],(X,\tau)\right)$.
	\item For every $\varepsilon>0$ and every $p\in\semis$ there exists $q\in\semis$ and $K>0$ such that for all $f\in\BC\left(\left[0,t_0\right],(X,\tau)\right)$
	\begin{align}
	p\left(\int_0^{t_0}{T_{-1}(t_0-r)Bf(r)\ \dd r}\right)\leq K\cdot\sup_{r\in\left[0,t_0\right]}{\left|q(f(r))\right|}+\varepsilon\left\|f\right\|_{\infty}.
	\end{align}
	\item There exists $M\in\left(0,1\right)$ such that for all $f\in\BC\left(\left[0,t_0\right],(X,\tau)\right)$ 
	\begin{align}
	\left\|\int\limits_0^{t_0}{T_{-1}(t_0-r)Bf(r)\ \dd r}\right\|\leq M\left\|f\right\|_{\infty}.
	\end{align}
\end{abc}
\noindent Then the operator $(A_{-1}+B)_{|X}$ defined on the domain 
\[
\dom((A_{-1}+B)_{|X}):=\left\{x\in X:\ A_{-1}x+Bx\in X\right\}
\]
generates a $\tau$-bi-continuous semigroup.
\end{theorem}

Before we start with the proof of the Theorem \ref{prop:PosPertAux} we have to think about approximations of continuous functions with values in locally convex spaces as this is needed for the application of Theorem \ref{thm:admDS}. For that recall the following definitions, cf. \cite[Def.~2.2(ii) \& (iii)]{ER1991}.

\begin{definition}
Let $(\Omega,\Sigma,\mu)$ be a measure space and $(X,\tau)$ a locally convex space with a family $\semis$ of seminorms generating $\tau$. We say that $f:\Omega\to(X,\tau)$ is \emph{measurable in seminorms} if for each $p\in\semis$ there exists a sequence $(f^{(p)}_{n})_{n\in\NN}$ of step functions such that for $\mu$-almost everywhere $x\in\Omega$ holds that
\[
\lim_{n\to\infty}{p(f^{(p)}_{n}(x)-f(x))}=0.
\]
\end{definition}

\begin{definition}
Let $(\Omega,\Sigma,\mu)$ be a measure space and $(X,\tau)$ a locally convex space with a family $\semis$ of seminorms generating $\tau$. We say that $f:\Omega\to(X,\tau)$ is \emph{weakly measurable} if for each $\varphi\in(X,\tau)'$ the function $x\to\left\langle \varphi,f(x)\right\rangle$ is measurable.
\end{definition}

\begin{remark}\label{rem:ApproxCont}
\begin{abc}
	\item Observe that by \cite[Prop.~2.3]{B1981} the notions of measurability in seminorms and weak measurability coincide if $X$ is a locally convex Suslin space.
	\item Since for the application of Theorem \ref{thm:admDS} only continuous functions $u:\left[0,t_0\right]\to(X,\tau)$ are important and $\varphi\in(X,\tau)'$ we conclude that $\left\langle \varphi,u(\cdot)\right\rangle$ is continuous and hence measurable. Moreover, $u\left(\left[0,t_0\right]\right)$ is separable since $\left[0,t_0\right]$ is separable. 
	\item We conclude that $u$ can be approximated by a $\left\|\cdot\right\|$-bounded sequence of step functions in seminorms.
\end{abc}
\end{remark}

\begin{proofof}{Theorem \ref{prop:PosPertAux}}
First assume that $0\in\rho(A)$ and that $(T(t))_{t\geq0}$ is exponentially stable, i.e., there exists $M\geq1$ and $\omega>0$ such that $\left\|T(t)\right\|\leq M\ee^{-\omega t}$ for all $t\geq0$ (for more characterizations we refer to \cite[Chapter III, Sect.~2]{EisnerStab}). Let $\varepsilon>0$ be arbitrary and $u(t)=\sum_{n=1}^N{x_n\textbf{1}_{\Omega_n}(t)}$ be a $\left\|\cdot\right\|$-bounded positive step function. Fix $t_0>0$ and define
\[
\R u:=\int_0^{t_0}{T_{-1}(t_0-s)Bu(s)\ \dd{s}}.
\]
Obviously, one has that $\R u\in X$ and $\R u\geq0$ by Lemma \ref{lemma:StepFunctions}. We define $z:=\sup_{n}{x_n}$ and obtain:
\begin{align*}
\left\|\R u\right\|&=\left\|\int_0^{t_0}{T_{-1}(t_0-s)Bu(s)\ \dd{s}}\right\|\leq\left\|\int_0^{t_0}{T_{-1}(t_0-s)Bz\ \dd{s}}\right\|\leq\left\|\int_0^{\infty}{T_{-1}(t_0-s)Bz\ \dd{s}}\right\|\\
&=\left\|A_{-1}^{-1}Bz\right\|\leq\left\|A_{-1}^{-1}B\right\|\cdot\left\|z\right\|\leq K\left\|z\right\|=K\left\|\sup_{n}{x_n}\right\|=K\sup_{n}{\left\|x_n\right\|}=K\left\|u\right\|_{\infty}.
\end{align*}
The last equality follows since by assumption $\left(X,\left\|\cdot\right\|,\tau\right)$ is a bi-$\mathrm{AM}$-space. For $p\in\semis$ we obtain by using that the operator $A_{-1}^{-1}B$ is local and $\varepsilon>0$ is arbitrary that there exists $q\in\semis$ such that
\begin{align*}
p(\R u)&=p\left(\int_0^{t_0}{T_{-1}(t_0-s)Bu(s)\ \dd{s}}\right)\leq p\left(\int_0^{t_0}{T_{-1}(t_0-s)Bz\ \dd{s}}\right)\leq p\left(\int_0^{\infty}{T_{-1}(t_0-s)Bz\ \dd{s}}\right)\\
&=p(A_{-1}^{-1}Bz)\leq Mq(z)+\varepsilon\left\|z\right\|=Mq(z)+\varepsilon\left\|u\right\|_{\infty}=Mq(\sup_n{x_n})+\varepsilon\left\|u\right\|_{\infty}\\
&=M\cdot\sup_n{q(x_n)}+\varepsilon\left\|u\right\|_{\infty}=M\sup_{t\in\left[0,t_0\right]}{q(u(t))}+\varepsilon\left\|u\right\|_{\infty}.
\end{align*}
Observe, that the second inequality follows again from Lemma \ref{lemma:StepFunctions} and that we make use of the assumption that $\left(X,\left\|\cdot\right\|,\tau\right)$ is a bi-$\mathrm{AM}$-space. We see that the conditions of the Theorem \ref{thm:admDS} are satisfied for positive step functions. If $u$ is now an arbitrary step function, then $u$ can be split in its positive part $u_+$ and its negative part $u_-$, i.e., $u=u_+-u_-$, and we obtain $\left|\R u\right|\leq\R\left|u\right|$ meaning that $\left\|\R u\right\|\leq K\left\|u\right\|_{\infty}$ and $p(\R u)\leq M\sup_{t\in\left[0,t_0\right]}{q(u(t))}+\varepsilon\left\|u\right\|_{\infty}$. Now let $u\in\BC\left(\left[0,t_0\right],(X,\tau)\right)$ be arbitrary and find for $p\in\semis$ a sequence $(u_n)_{n\in\NN}$ of step functions approximating $u$. Since by assumption $B$ is a bounded linear operator between Banach spaces and in addition a linear and continuous map between locally convex spaces and $(T_{-1}(t))_{t\geq0}$ is a bi-continuous semigroup we conclude by means of the norming property of $X$, i.e., Assumption \ref{asp:bicontspace}(3), that $\R u_n\to\R u$, that $\int_0^{t_0}{T_{-1}(t_0-s)Bu(s)\ \dd{s}}=\R u\in X$ and
\[
\left\|\int_0^{t_0}{T_{-1}(t_0-s)Bu(s)\ \dd{s}}\right\|\leq K\left\|u\right\|_{\infty},
\] 
as well as 
\[
p\left(\int_0^{t_0}{T_{-1}(t_0-s)Bu(s)\ \dd{s}}\right)\leq M\sup_{t\in\left[0,t_0\right]}{q(u(t))}+\varepsilon\left\|u\right\|_{\infty}.
\]
This shows that we can apply Theorem \ref{thm:admDS} and conclude that $(A_{-1}+B)_{|X}$ generates a bi-continuous semigroup $(S(t))_{t\geq0}$ on $X$. By recapping the proof of Theorem \ref{thm:admDS} we see that $(S(t))_{t\geq0}$ is actually given by a Dyson--Phillips series, i.e., for all $t\geq0$ one has
\[
S(t)=\sum_{n=0}^{\infty}{S_n(t)},
\]
where $S_0(t):=T(t)$ and
\[
S_n(t)x=\int_0^t{T_{-1}(t-s)BS_{n-1}(s)x\ \dd{s}},
\]
for all $x\in X$. From this we conclude the positivity of $(S(t))_{t\geq0}$. 

\medskip
For the general case, we observe that $(A,\dom(A))$ generates a positive bi-continuous semigroup $(T(t))_{t\geq0}$ if and only if $(A-\lambda,\dom(A))$ generates a positive bi-continuous semigroup $(\ee^{-\lambda t}T(t))_{t\geq0}$ for $\lambda>\mathrm{s}(A)$. Since $\lambda\mapsto\left\|\res(\lambda,A_{-1})B\right\|$ is decreasing as a function on $(\mathrm{s}(A),\infty)$ we can find $\lambda\in\RR$ such that that $A-\lambda$ is the generator of a positive exponentially stable bi-continuous semigroup such that $\left\|\res(\lambda,A_{-1})B\right\|<1$. The first part of the current proof then implies that $(A_{-1}-\lambda+B)_{|X}$ generates a positive bi-continuous semigroup on $X$ and we are done.
\end{proofof}

\begin{remark}
\begin{abc}
	\item Notice, that it is crucial for the proof of Theorem \ref{prop:PosPertAux} that we assume that the underlying space is a bi-$\mathrm{AM}$-space and that the operator $R(\lambda,A_{-1})$ is local. Otherwise, one can not get the required estimates one needs to apply Theorem \ref{thm:admDS}.
	\item We observe that the extrapolation spaces of the original semigroup $(T(t))_{t\geq0}$ and the perturbed semigroup $(S(t))_{t\geq0}$ are the same, see also \cite[Prop.~4.2]{BJVW2018}.
\end{abc}
\end{remark}

The next result directly follows from Theorem \ref{prop:PosPertAux}.

\begin{corollary}\label{thm:MainPosDS}
Let $\left(X,\left\|\cdot\right\|,\tau\right)$ be a bi-$\mathrm{AM}$-space and $(T(t))_{t\geq0}$ a positive bi-continuous semigroup on $X$ with generator $(A,\dom(A))$. Let $B:X\to X_{-1}$ be a positive operator such that $B:(X,\tau)\to(X,\tau_{-1})$ is linear and continuous. Suppose that there exists $\lambda>s(A)$ such that $\res(\lambda,A_{-1})B$ is local and $\mathrm{r}\left(\res(\lambda,A_{-1})B\right)<1$. Then $(A_{-1}+B)_{|X}$ is the generator of a positive bi-continuous semigroup $(S(t))_{t\geq0}$.
\end{corollary}

\begin{proof}
Since by assumption $\mathrm{r}\left(\res(\lambda,A_{-1})B\right)<1$ we obtain that $\lambda\in\rho(A_{-1}+B)$ and $\res(\lambda,A_{-1}+B)\geq0$, i.e., $A_{-1}+B$ is a resolvent positive operator. In addition, for all $s\in\left(0,1\right)$, one gets that $\lambda\in\rho(A_{-1}+sB)$ and
\[
\res(\lambda,A_{-1})\leq\res(\lambda,A_{-1}+sB)\leq\res(\lambda,A_{-1}+B).
\]
Now choose $n\in\NN$ such that $\left\|\res(\lambda,A_{-1}+B)B\right\|<n$ which means that $\left\|\frac{1}{n}\res\left(\lambda,A_{-1}+\frac{j}{n}B\right)B\right\|<1$ for all $1\leq j\leq n$. Now we apply Theorem \ref{prop:PosPertAux} to the operators 
\[
A,\left(A_{-1}+\frac{1}{n}B\right)_{|X},\ldots,\left(A_{-1}+\frac{n-1}{n}B\right)_{|X},
\]
with the perturbation operator $\frac{1}{n}B$ and the proof is finished.
\end{proof}

\section{Examples}
\subsection{The translation semigroup on $\BC(\RR)$}
First of all we discuss an example on the space of bounded continuous function $\BC(\RR)$ where we apply our theory from the previous section. In particular, we consider the following initial value problem for the partial differential equation on $\BC(\RR)$ given by
\begin{align*}\tag{PDE}\label{eqn:PDE}
\begin{cases}
\displaystyle{\frac{\partial}{\partial{t}}w(t,x)}=\displaystyle{\frac{\partial}{\partial{x}}w(t,x)+\int_\RR{w(t,\xi)\ \dd\mu(\xi)}\cdot\chi_{\left[1,\infty\right)}(x)},&\quad t\geq0,\ x\in\RR\\
w(0,x)=u_0(x),&\quad t\in\RR
\end{cases}
\end{align*}
where $\mu$ is a regular bounded Borel measure in $\RR$ satisfying $\left|\mu\right|(\RR)<1$. To continue, we rewrite \eqref{eqn:PDE} as an abstract Cauchy problem on $\BC(\RR)$. For that reason we consider the space $X:=\BC(\RR)$ with the sup-norm and the compact-open topology $\tau=\tau_{\mathrm{co}}$ as additional locally convex topology. Recall that $\tau$ is generated by the family of seminorms $\semis=\semis_{\mathrm{co}}=\left\{p_K:\ K\subseteq\RR\ \text{compact}\right\}$ where $p_K(f):=\sup_{x\in K}{\left|f(x)\right|}$.
{By Example \ref{exa:BiContSpaces}(a), the space $\BC(\RR)$ equipped with the sup-norm and the compact-open topology satisfies Assumption \ref{asp:bicontspace}. Moreover, by \cite[Thm.~7.4, Thm.~8.1 \& Cor.~8.2]{Kr2016} also Assumption \ref{ass:CompMixed} is satisfied.} As a matter of fact, $\BC(\RR)$ is a Banach lattice equipped with the ordering defined by $f\geq g$ if and only if $f(x)\geq g(x)$ for all $x\in\RR$. It is not hard to see that $\tau$ is compatible with the Banach lattice structure of $X$ and that it is a bi-$\mathrm{AM}$-space, too. Let $(T(t))_{t\geq0}$ be the left-translation semigroup on $\BC(\RR)$, i.e.,
\[
(T(t)f)(x):=f(x+t),\quad t\geq0,\ f\in\BC(\RR),\ x\in\RR,
\]
with generator $(A,\dom(A))$ given by $A:=\frac{\dd}{\dd{x}}$ and $\dom(A):=\BC^1(\RR)$. Obviously, $(T(t))_{t\geq0}$ is a positive operator semigroup. We now consider an operator $B$ which is related to the one in \cite[Sect.~4]{BF2}. {Let the function $h:\RR\to\RR$ be defined by
\[
h(x):=\begin{cases}
\ee^{x-1},&\quad x\leq1,\\
1,&\quad x>1.
\end{cases}
\]
and consider the function $g:\RR\to\RR$ defined by
\[
g(x):=\chi_{\left[1,\infty\right)}(x)=\begin{cases}
0,&\quad x\leq1,\\
1,&\quad x>1.
\end{cases}
\]
Recall from \cite[Sect.~6.1]{BF} that the first extrapolation space $X_{-1}$ corresponding the the left-translation semigroup is given by
\[
X_{-1}=\left\{F\in\mathscr{D}'(\RR):\ F=f-Df\ \text{for some}\ f\in\BC(\RR)\right\},
\]
where $\mathscr{D}'(\RR)$ denotes the space of all distributions on $\RR$ and  $D: \mathscr{D}'(\RR)\to\mathscr{D}'(\RR)$ the distributional derivative. Observe, that $g\in X_{-1}$ since $g=h-Dh$.} For the functional $\Phi:\BC(\RR)\to\CC$ defined by $\Phi(f):=\int_{\RR}{f\ \dd\mu}$ we define $B:X\to X_{-1}$ by
\[
Bf:=\Phi(f)g=\int_{\RR}{f\ \dd\mu}\cdot g,
\]
where $\mu$ is the bounded regular Borel measure on $\RR$ from above. Having this, \eqref{eqn:PDE} can be rewritten as
\[
\begin{cases}
\dot{u}(t)=Au(t)+Bu(t),&\quad t\geq0\\
u(0)=u_0.&
\end{cases}
\] 
First of all we show that $B$ is a positive operator, i.e., that $Bf\geq0$ whenever $f\geq0$. To do so, it suffices to show that there exists a sequence $(h_n)_{n\in\NN}$ of positive elements in $\BC(\RR)$ such that $h_n\stackrel{\tau_{-1}}{\rightarrow}g$. Define a sequence $(h_n)_{n\in\NN}$ of positive elements in $\BC(\RR)$ by
\[
h_n(x):=\begin{cases}
0,&\quad x\leq0,\\
x^n,&\quad 0<x\leq1,\\
1,&\quad x>1.
\end{cases}
\]
We continue by determining the expression
\[
(\res(1,A)h_n)(x)=\int_0^{\infty}{\ee^{-t}h_n(x+t)\ \dd{t}}.
\]
Due to Remark \ref{rem:EquivExtraSp} it suffices to consider $\res(1,A)$ instead of $A^{-1}$. To do so, we have to distinguish three cases. For $x\leq0$ we have
\[
\int_0^{\infty}{\ee^{-t}h_n(x+t)\ \dd{t}}=\int_{-x}^{1-x}{\ee^{-t}h_n(x+t)\ \dd{t}}+\int_{1-x}^{\infty}{\ee^{-t}\ \dd{t}}=\ee^x\left(\Gamma(n+1)-\Gamma(n+1,1)\right)+\ee^{x-1},
\]
for $0\leq x\leq1$ one gets
\[
\int_0^{\infty}{\ee^{-t}h_n(x+t)\ \dd{t}}=\int_{0}^{1-x}{\ee^{-t}(x+t)^n\ \dd{t}}+\int_{1-x}^{\infty}{\ee^{-t}\ \dd{t}}=\ee^x\left(\Gamma(n+1,x)-\Gamma(n+1,1)\right)+\ee^{x-1},
\]
whereas for $x>1$
\[
\int_0^{\infty}{\ee^{-t}h_n(x+t)\ \dd{t}}=\int_0^{\infty}{\ee^{-t}\ \dd{t}}=1.
\]
Here $\Gamma(x)$ denotes the Gamma function and $\Gamma(a,x)$ the incomplete Gamma function. For the sake of completeness, we recall the definition of the incomplete Gamma function, since this does not find much space is the literature. By definition, for $a>0$ and $x\geq0$ the incomplete Gamma function $\Gamma(a,x)$ is defined by
\[
\Gamma(a,x):=\int_x^\infty{t^{a-1}\ee^{-t}\ \dd{t}},
\]
and has been studied extensively by Jameson \cite{J2016}. As a matter of fact, by \cite[Thm.~3]{J2016} one has that for $n\in\NN$ and $x\geq0$
\[
\Gamma(n+1,x)=n!\ee^{-x}\sum_{m=0}^n{\frac{x^m}{m!}}.
\]
In particular, one has $\Gamma(n+1)-\Gamma(n+1,1)=n!-\frac{\left\lfloor \ee n!\right\rfloor}{\ee}$ and $\lim_{n\to\infty}{\Gamma(n+1,x)-\Gamma(n+1,1)}=0$ for $x\geq0$. Combined with the expressions before we conclude that
\[
\sup_{x\in K}{\left|h(x)-\int_0^{\infty}{\ee^{-t}h_n(x+t)\ \dd{t}}\right|}\to0\quad (n\to\infty),
\]
for each compact $K\subseteq\RR$, which in fact shows that $h_n\stackrel{\tau_{-1}}{\rightarrow}g$. Furthermore, by construction one has that $\sup_{n\in\NN}{\left\|h_n\right\|}<\infty$ and hence we may conclude that $g\in X_{-1,+}$, showing that $B$ is in fact a positive operator. In order to apply Theorem \ref{prop:PosPertAux} we have to show that $\res(1,A_{-1})B$ is local and $\mathrm{r}\left(\res(1,A_{-1})B\right)<1$. Observe that by construction one has that
\[
\res(1,A_{-1})Bf=\res(1,A_{-1})\Phi(f)g=\Phi(f)\res(1,A_{-1})g=\Phi(f)h.
\]
By the regularity of $\mu$ we find for $\varepsilon>0$ arbitrary a compact set $K'\subseteq\RR$, such that $\left|\mu\right|(\RR\setminus K')<\varepsilon$. Hence for all $f\in\BC(\RR)$
\begin{align*}
\left|\res(1,A_{-1})Bf(x)\right|&=\left|\Phi(f)h(x)\right|\\
&=\left|\int_\RR{f(y)\ \dd\mu(y)}\cdot h(x)\right|\\
&\leq\int_{\RR}{\left|f(y)\right|\ \dd\left|\mu\right|(y)}\cdot\left|h(x)\right|\\
&=\left(\int_{K'}{\left|f(y)\right|\ \dd\left|\mu\right|(y)}+\int_{\RR\setminus K'}{\left|f(y)\right|\ \dd\left|\mu\right|(y)}\right)\cdot\left\|h\right\|\\
&\leq\left(\left|\mu\right|(K')\sup_{x\in K'}{\left|f(x)\right|}+\varepsilon\left\|f\right\|\right)\cdot\left\|h\right\|,
\end{align*}
showing that $\res(1,A_{-1})B$ is indeed local. On the other hand
\[
\left|\res(1,A_{-1})Bf(x)\right|=\left|\Phi(f)h(x)\right|\leq\left|\mu\right|(\RR)\left\|f\right\|\left\|h\right\|,
\]
showing that $\left\|\res(1,A_{-1})B\right\|<1$ if $\left|\mu\right|(\RR)<1$. Hence Theorem \ref{prop:PosPertAux} applies and we may conclude that $A_{-1}+B$ generates a positive semigroup.

\subsection{The left-implemented semigroup on $\LLL(E,F)$}
Let us consider bi-continuous semigroups on $\LLL(E,F)$, the space of bounded linear operator from a Banach space $E$ to another Banach space $F$. By Example \ref{exa:BiContSpaces}(b) the space $\LLL(E,F)$ equipped with the strong operator topology satisfies Assumption \ref{asp:bicontspace}. We will treat this space and topology in more detail. The so-called \emph{strong operator topology} $\tau_\mathrm{sot}$ on $\LLL(E,F)$ is generated by the family of seminorms $\semis=\left\{p_x:\ x\in E\right\}$, where $p_x(T):=\left\|Tx\right\|$ for $T\in\LLL(E,F)$. Notice that $\LLL(E,F)$ does not yield a Banach lattice structure and much less a $\mathrm{AM}$-space structure in general, even if $E$ and $F$ are Banach lattices. However, under certain circumstances this happens. For that, we recall the following notion, cf. \cite[Chapter IV, Def.~4.2]{Schaefer1974}.

\begin{definition}
A Banach lattice $(E,\left\|\cdot\right\|,\leq)$ is said to have \emph{property ($\mathrm{P}$)} if there exists a positive, contractive projection $E''\to E$, where $E$ (under evaluation) is identified with a vector sublattice of its bidual $E''$.
\end{definition}

In what follows, we also need the following concept.

\begin{definition}
A Banach lattice $(E,\left\|\cdot\right\|,\leq)$ is called \emph{$\mathrm{AL}$-space} if $\left\|x+y\right\|=\left\|x\right\|+\left\|y\right\|$ for all $x,y\geq0$.
\end{definition}

We remark, that every $\mathrm{KB}$-space, every $\mathrm{AL}$-space and every order complete $\mathrm{AM}$-space with unit has property ($\mathrm{P}$), cf. \cite[Chapter IV, Sect.~4]{Schaefer1974}. The following will be crucial for our set-up.

\begin{hypo*}
Let $(E,\left\|\cdot\right\|,\leq)$ and $(F,\left\|\cdot\right\|,\leq)$ be Banach lattices such that
\begin{abc}
	\item $F$ is an $\mathrm{AM}$-space and has property ($\mathrm{P}$),
	\item $E$ is an $\mathrm{AL}$-space.
\end{abc}
\end{hypo*}

By \cite[Chapter IV, Prop.~4.4]{Schaefer1974} we may conclude that under this hypothesis $\LLL(E,F)$ becomes an (order complete) $\mathrm{AM}$-space with unit. We claim that $\LLL(E,F)$ in this situation is even a bi-$\mathrm{AM}$-space. For that notice that in our situation, $\tau_{\mathrm{sot}}$ is generated by a smaller family of seminorms, namely the family $\semis_+:=\left\{p_x:\ x\in E_+\right\}$, since every element $x\in E$ can be decomposed as $x=x_+-x_-$ for $x_+,x_-\in E_+$. Let us call the locally convex topology generated by $\semis_+$ by $\tau_\mathrm{sot}^+$ and observe that $\tau_\mathrm{sot}$ and $\tau_\mathrm{sot}^+$ coincide. Equipped with $\tau^+_{\mathrm{sot}}$, the space $\LLL(E,F)$ also satisfies the general Assumptions \ref{asp:bicontspace} and becomes a bi-$\mathrm{AM}$-space due to the assumption that $F$ is an $\mathrm{AM}$-space.

\medskip
Let us now turn to operator semigroups. For a fixed and positive $C_0$-semigroup $(T(t))_{t\geq0}$ on $F$ with generator $(A,\dom(A))$, the so-called \emph{left-implemented semigroup} $(\mathcal{U}(t))_{t\geq0}$ on $\LLL(E,F)$ is defined by
\[
\mathcal{U}(t)S:=T(t)S,\quad t\geq0,\ S\in\LLL(E,F).
\]
This semigroup as well as its extrapolation spaces have been intensively studied by J.~Alber \cite{AlberDip,Alber2001} and more recently by B.~Farkas and the author \cite{BF,BF2,BuddePhD}. This semigroups is known to be bi-continuous with respect to the strong operator topology $\tau_{\mathrm{sot}}$. Furthermore, we observe that $(\mathcal{U}(t))_{t\geq0}$ becomes a positive semigroup since $(T(t))_{t\geq0}$ was assumed to be positive. Since $\tau_\mathrm{sot}$ is a complete locally convex topology if restricted to the closed unit ball of $\LLL(E,F)$, we conclude by \cite[Chapter I, Prop.~1.14]{C1987} that the space $\LLL(E,F)$ is $\gamma$-complete. Moreover, one has that for $(a_n)_{n\in\NN}\in\mathrm{c}_0$, $a_n\geq0$, and $(x_n)_{n\in\NN}\subseteq E$
\[
\sup_{t\in\left[0,t_0\right]}\widetilde{p}_{(a_n,p_{x_n})}(\mathcal{U}(t)S)=\sup_{t\in\left[0,t_0\right]}{\sup_{n\in\NN}a_n\left\|\mathcal{U}(t)Sx_n\right\|}\leq\sup_{t\in\left[0,t_0\right]}M\ee^{\omega t}\cdot{\sup_{n\in\NN}a_n\left\|Sx_n\right\|}\leq K\cdot\widetilde{p}_{(a_n,p_{x_n})}(S),
\]
showing that $(\mathcal{U}(t))_{t\geq0}$ is $\gamma$-locally equicontinuous. Hence, also Assumption \ref{ass:CompMixed} is satisfied. By \cite[Sect.~6.4]{BF}, we know that the generator $(\mathcal{G},\dom(\mathcal{G}))$ of $(\mathcal{U}(t))_{t\geq0}$ is given by $\mathcal{G}S=A_{-1}S$ on the domain $\dom(\mathcal{G})=\left\{S\in\LLL(E,F):\ A_{-1}S\in\LLL(E,F)\right\}$. Here, $A_{-1}$ denotes the extrapolated generator of $(T(t))_{t\geq0}$. Again by \cite[Sect.~6.4]{BF}, we know exactly, how the extrapolation space $X_{-1}$ of the left-implemented semigroup $(\mathcal{U}(t))_{t\geq0}$ looks like, i.e., one has that $X_{-1}=\LLL(E,F_{-1})$, where $F_{-1}$ is the first extrapolation space with respect to the $C_0$-semigroup $(T(t))_{t\geq0}$. 

\medskip
Now we are able to consider positive Desch--Schappacher perturbations. To do so, let $B:F\to F_{-1}$ be a positive operator and assume for simplicity that there exists $\lambda>s(A)$ with $\left\|R(\lambda,A_{-1})B\right\|<1$. Since $F$ is an $\mathrm{AM}$-space by assumption from above, we conclude by \cite[Thm.~1.2 \& Prop.~4.2]{BJVW2018} that $(A_{-1}+B)_{|F}$ generates a positive $C_0$-semigroup $(S(t))_{t\geq0}$ on $F$. Similar to \cite[Thm.~5.5]{BF2}, we define an operator $\mathcal{K}:\LLL(E,F)\to\LLL(E,F_{-1})$ by
\[
\mathcal{K}S:=BS,\quad S\in\LLL(E,F).
\]
We know by the same theorem, cf. \cite[Thm.~6.5]{BF2}, that $(\mathcal{G}_{-1}+\mathcal{K})_{|\LLL(E,F)}$ generates a bi-continuous semigroup on $\LLL(E,F)$ which is implemented by $(S(t))_{t\geq0}$ and hence is again positive due to the positivity of $(S(t))_{t\geq0}$. However, we want to show that this already follows by Theorem \ref{prop:PosPertAux} without making use of the results by Batkai et al. \cite[Thm.~1.2 \& Prop.~4.2]{BJVW2018}. For that purpose, we first have to show that $\mathcal{K}$ is a positive operator. To do so, let $S\in\LLL(E,F)$ be positive. To show that $\mathcal{K}$ is positive, choose $x\in E_+$ arbitrary and observe that $Sx\in F_+$ and that $B(F_+)\subseteq F_{-1,+}$ showing that $BSx\geq0$, hence $BS\geq0$ and therefore $\mathcal{K}$ is positive. Owing to the fact that one has
\[
R(\lambda,\mathcal{G}_{-1})\mathcal{K}Sx=R(\lambda,A_{-1})BSx,\quad S\in\LLL(E,F),\ x\in E,\ \lambda>s(A),
\]
it is obvious that the operator $R(\lambda,\mathcal{G}_{-1})\mathcal{K}$ is local and satisfies $\left\|R(\lambda,\mathcal{G}_{-1})\mathcal{K}\right\|<1$ since we also assumed that  $\left\|R(\lambda,A_{-1})B\right\|<1$. Hence, we can apply Theorem \ref{prop:PosPertAux} to conclude that $(\mathcal{G}_{-1}+\mathcal{K})_{|\LLL(E,F)}$ generates a positive bi-continuous semigroup $(\mathcal{V}(t))_{t\geq0}$ on $\LLL(E,F)$. That $(\mathcal{V}(t))_{t\geq0}$ is implemented, too, one has to consider \cite[Lemma~6.3]{BF2}.

\section*{Acknowledgement}
The author is indepted to B\'{a}lint Farkas and Sanne ter Horst for the fruitful discussions, the continuous support during writing and the helpful feedback. This study was funded by the DAAD-TKA Project 308019 ``\emph{Coupled systems and innovative time integrators}''. 

\end{document}